\documentclass[12pt, reqno]{amsart}
\usepackage{amsmath, amsthm, amscd, amsfonts, amssymb, graphicx, color}
\usepackage[bookmarksnumbered, colorlinks, plainpages]{hyperref}

\input{mathrsfs.sty}

\newtheorem{thm}{Theorem}[section]
\newtheorem{lem}[thm]{Lemma}
\newtheorem{cor}[thm]{Corollary}

\newtheorem{rem}[thm]{Remark}

\theoremstyle{definition}
\newtheorem{defn}{Definition}[section]
\theoremstyle{definition}

\theoremstyle{remark}

\theoremstyle{question}

\numberwithin{equation}{section}

\begin{document}

\title[Halmos' two projections theorem]{Halmos' two projections theorem for Hilbert $C^*$-module operators and the Friedrichs} 
\author{Wei Luo}
\address{Department of Mathematics, Shanghai Normal University, Shanghai 200234, PR China}
\email{luoweipig1@163.com}
\author{Mohammad Sal Moslehian}
\address{Department of Pure Mathematics, Center of Excellence in
Analysis on Algebraic Structures (CEAAS), Ferdowsi University of
Mashhad, P. O. Box 1159, Mashhad 91775, Iran.}
\email{moslehian@yahoo.com; moslehian@um.ac.ir}
\author{Qingxiang Xu}
\address{Department of Mathematics, Shanghai Normal University, Shanghai 200234, PR China}
\email{qingxiang$\_$xu@126.com}
\address{Department of Pure Mathematics, Ferdowsi University of Mashhad, P.O. Box 1159, Mashhad 91775, Iran}



\begin{abstract} Halmos' two projections theorem for Hilbert space operators is one of the fundamental results in operator theory. In this paper, we introduce
the term of two harmonious projections in the context of adjointable operators on Hilbert $C^*$-modules, extend Halmos' two projections theorem to the case of two harmonious projections. We also give some new characterizations of the closed submodules and their associated projections. As an application, a norm equation associated to a characterization of the Friedrichs angle is proved to be true in the framework of Hilbert $C^*$-modules.
\end{abstract}
\maketitle
\section{Introduction}
Let $M$ and $N$ be two closed subspaces of a Hilbert space $H$. The Friedrichs angle \cite{Friedrichs}, denoted by $\alpha(M,N)$, is the unique angle in $[0,\frac{\pi}{2}]$ whose cosine is equal to $c(M,N)$, where
\begin{eqnarray*}c(M,N)=\sup\big\{ |\langle x,y\rangle |: x\in \widetilde{M}, y\in \widetilde{N}, \Vert x\Vert\le 1, \Vert y\Vert\le 1\big\},
\end{eqnarray*}
in which $\widetilde{M}=M\cap (M\cap N)^\bot$ and $\widetilde{N}=N\cap (M\cap N)^\bot$. It is known \cite{Deutsch} that
\begin{equation}\label{equ:computation of C M N}c(M,N)=\big\Vert P_MP_N\big(I-P_{M\cap N}\big)\Vert,\end{equation}
where the notation $P_E$ stands for the projection from $H$ onto its closed subspace $E$. Given any projection $P$ on $H$, let $\mathcal{R}(P)$ and $\mathcal{N}(P)$ denote the range and the null space of $P$, respectively. It is proved in \cite{Deutsch} that for every two projections $P$ and $Q$ on $H$,
\begin{equation}\label{equ:motivation equation}\big\Vert PQ\big(I-P_{\mathcal{R}(P)\cap \mathcal{R}(Q)}\big) \big\Vert=\big\Vert (I-P)(I-Q)\big(I-P_{\mathcal{N}(P)\cap \mathcal{N}(Q)}\big)\big\Vert,
\end{equation}
which gives a characterization of the Friedrichs angle as
\begin{equation}\label{equ:M,N norm equivalent for the angle-1}c(M,N)=c(M^\perp, N^\perp).\end{equation}
The proof of equation \eqref{equ:motivation equation} given in \cite{Deutsch} relies on the Pythagorean theorem, that is,
$$\|x\|^2=\|P(x)\|^2+\|(I-P)x\|^2,$$
where $P$ is any projection on $H$ and $x\in H$ is arbitrary.

The Hilbert $C^*$-module is the generalization of the Hilbert space by allowing the inner product to take values in certain $C^*$-algebra $A$ instead of the complex field $\mathbb{C}$. The purpose of this paper is to investigate the validity of \eqref{equ:motivation equation} in the Hilbert $C^*$-module case. Let $H$ be a Hilbert $C^*$-module and $M$ be a closed submodule of $H$. It can be verified directly that the notation $P_M$ is meaningful if and only if
$M$ is orthogonally complemented in $H$, and in this case $(M^\bot)^\bot=M$. So, for projections $P$ and $Q$ on $H$, one can associate the determination of the orthogonal complementarity of $\mathcal{R}(P)\cap \mathcal{R}(Q)$ and $\mathcal{N}(P)\cap \mathcal{N}(Q)$ to \eqref{equ:motivation equation}.

Let $P$ and $Q$ be two projections on a Hilbert $C^*$-module. It is clear that $\overline{\mathcal{R}(P+Q)}^\bot=\mathcal{N}(P)\cap \mathcal{N}(Q)$, so $P_{\mathcal{N}(P)\cap \mathcal{N}(Q)}$ is meaningful whenever $\overline{\mathcal{R}(P+Q)}$ is orthogonally complemented. This observation together with \eqref{equ:motivation equation} lead us to study such a topic: Assume that
$\overline{\mathcal{R}(P+Q)}$ is orthogonally complemented, under what conditions $\overline{\mathcal{R}(I-P+I-Q)}$ is also orthogonally complemented. In Section~\ref{sec:Orthogonal complementarity of closed submodules}, we will give some necessary and sufficient conditions on this topic. Another observation is that the orthogonal complementarity of $\mathcal{R}(P)\cap \mathcal{R}(Q)$ is obviously guaranteed
if $\mathcal{R}(P)\cap \mathcal{R}(Q)=\{0\}$. So we turn to consider the projections $P$ and $Q$ such that $\mathcal{R}(P)$ and $\mathcal{R}(Q)$ are in generic position \cite{Halmos}, that is, $P$ and $Q$ satisfy
$$\mathcal{R}(P)\cap\mathcal{R}(Q)=\mathcal{R}(P)\cap\mathcal{N}(Q)=\mathcal{N}(P)\cap\mathcal{R}(Q)=\mathcal{N}(P)\cap\mathcal{N}(Q)=\{0\}.$$
In Section~\ref{sec:discomplementable projections}, we have managed to construct such two projections $P$ and $Q$ on some Hilbert $C^*$-module $H$ such that neither
\begin{equation}\label{equ:four closures}\overline{\mathcal{R}(P+Q)}, \overline{\mathcal{R}(P+I-Q)}, \overline{\mathcal{R}(I-P+Q)}\ \mbox{nor}\ \overline{\mathcal{R}(I-P+I-Q)}
\end{equation}
is equal to $H$ ensuring that none of them is orthogonally complemented in $H$. We call such $P$ and $Q$ extremely discomplementable projections.

It is notable that the Pythagorean theorem is no longer true for a general Hilbert $A$-module $H$, since for a projection $P$ on $H$ and an element $x\in H$,
the associated two positive elements $a=\langle Px,x\rangle$ and $b=\langle (I-P)x,x\rangle$ in the given $C^*$-algebra $A$ will only satisfy the inequality
$\|a+b\|\le \|a\|+\|b\|$ rather than the equality $\|a+b\|=\|a\|+\|b\|$. This leads us to
study the validity of \eqref{equ:motivation equation} by constructing unitary operators based on the generalized Halmos' two projections theorem.

Halmos' two projections theorem \cite{Halmos} for Hilbert space operators is one of the fundamental results in operator theory, see \cite{ACL}, \cite{Foulis-Jencov¨¢-Pulmannov¨¢} and \cite{Uebersohn}. It says that if $P$ and $Q$ are two projections on a Hilbert space such that $\mathcal{R}(P)$ and $\mathcal{R}(Q)$ are in generic position,  then there exist a unitary map $W$ from $M_1=\mathcal{R}(P)$ onto $M_2=\mathcal{N}(P)$, and two positive operators $C$ and $D$ on $M_1$ with $C^2+D^2=I, CD=DC$ so that under the orthogonal decomposition $H=M_1\dotplus M_2$, we have
\[P=\left(\begin{array}{cc} I & 0 \\0 & 0 \end{array}\right) \ \mbox{and}\  Q=\left(\begin{array}{cc} C^2 & CDW^{-1} \\WCD & WD^2W^{-1} \end{array}\right)\,. \]
It has applications in many areas such as the $cs$-decomposition, characterizations of the closedness of the sum of two subspaces, derivations of von Neumann's formula and the Feldman--Krupnik--Markus formulas, as well as computations of various angles and gaps between two subspaces of a Hilbert space. For the details, the reader is referred to the excellent survey \cite{Bottcher-Spitkovsky} and the references therein; see also \cite{AC,BRO,SIM}.
As far as we know, little has been done in the generalization of Halmos' two projections theorem for Hilbert $C^*$-module operators, which is the  concern of this paper. In Section~\ref{sec:Halmos' two projections theorem}, we have made
some new characterizations of the closed submodules and the associated projections; see \eqref{equ:formula for the projection on QP}, \eqref{equorthogonal decomposition of the range of QP}, \eqref{make projection for}, \eqref{eqn:prepare for the polar decomposition-1} and \eqref{eqn:prepare for the polar decomposition-2}, respectively. Here the key point is the introduction of
the harmonious projections described in Definition~\ref{defn:harmonious projections}. It is proved in Theorem~\ref{thm:Halmos' two projections theorem} that Halmos' two projections theorem remains to be true for every two harmonious projections. As an application, in Section~\ref{sec:the Friedrichs angle} we  show
that equation \eqref{equ:motivation equation} is true firstly in Lemma~\ref{lem:application of Halmos two projections-1} for every two harmonious projections, secondly in Lemma~\ref{lem:the main result+2} for every two projections $P$ and $Q$ such that $\overline{\mathcal{R}(P+Q)}$ and $\overline{\mathcal{R}(2I-P-Q)}$
are both orthogonally complemented, and finally in Theorem~\ref{thm:the final main result} for every two projections $P$ and $Q$ whenever  $\mathcal{R}(P)\cap\mathcal{R}(Q)$ and $\mathcal{N}(P)\cap\mathcal{N}(Q)$ are both orthogonally complemented\footnote{Note that the notations
 $P_{\mathcal{R}(P)\cap \mathcal{R}(Q)}$ and $P_{\mathcal{N}(P)\cap \mathcal{N}(Q)}$ in  equation \eqref{equ:motivation equation}
are meaningful if and only if such an orthogonal complementarity condition is satisfied.}.
Overall, it helps us to extend our viewpoint of the geometry of Hilbert $C^*$-modules; see, e.g., \cite{FRA, MT1, MKX}

Let us briefly recall some basic knowledge about Hilbert $C^*$-modules and adjointable operators. An inner-product module over a $C^*$-algebra $A$ is a right $ A$-module $H$ equipped with an $ A$-valued inner product $\langle \cdot, \cdot \rangle: H \times H \to A$ that is $\mathbb{C}$-linear and $A$-linear in the second variable and satisfies $\langle x, y \rangle^*=\langle y, x\rangle$ as well as $\langle x, x\rangle \geq 0$ with equality if and only if $x = 0$. An inner-product $A$-module $H$ which is complete with respect to the induced norm $\Vert x\Vert=\sqrt{\Vert \langle x,x\rangle\Vert}\,\,(x\in H)$ is called a (right) Hilbert $A$-module.

Suppose that $H$ and $K$ are two Hilbert $A$-modules, let $\mathcal{L}(H,K)$ be the set of operators $T:H\to K$ for which there is an operator $T^*:K\to H$ such that $\langle Tx,y\rangle=\langle x,T^*y\rangle$ for all $x\in H$ and $y\in K$. Each member in $\mathcal{L}(H,K)$ is called an adjointable operator. When $H=K$, $\mathcal{L}(H,H)$, abbreviated to $\mathcal{L}(H)$, is a $C^*$-algebra. By a positive operator we mean an operator $T\in \mathcal{L}(H)$ such that $\langle Tx,x\rangle\geq 0$ for all $x\in H$ \cite[Lemma~4.1]{Lance}. The strict topology (strong$^*$ topology) on $\mathcal
{L}(H,K)$ is defined to be the topology determined by the seminorms $T\mapsto \|Tx\|\,\,(x\in H)$ and $T\mapsto \|T^*y\|\,\,(y\in K)$. For every $T\in {\mathcal L}(H,K)$, the range and the null space of $T$ are designated by $\mathcal{R}(T)$ and $\mathcal{N}(T)$, respectively. By $I_H$ (or simply $I$) we denote the identity operator on $H$. The reader is referred to \cite{Lance,MT2} for some other basic notions related to Hilbert $C^*$-modules.

In this paper, the notations of ``$\oplus$" and ``$\dotplus$" are used
with different meanings for the sake of reader's convenience. For Hilbert $A$-modules
$H_1$ and $H_2$, let
$$H_1\oplus H_2:=\left\{(h_1, h_2)^T :h_i\in H_i, i=1,2\right\},$$ which is also a Hilbert $A$-module whose $A$-valued inner product is given by
$$\left<(x_1, y_1)^T, (x_2, y_2)^T\right>=\big<x_1,x_2\big>+\big<y_1,
y_2\big>$$ for $x_i\in H_1, y_i\in H_2, i=1,2$. If both
$H_1$ and $H_2$ are submodules of a Hilbert $A$-module $H$ such that $H_1\cap H_2=\{0\}$, then we
set
$$H_1\dotplus H_2:=\{h_1+h_2 : h_i\in H_i, i=1,2\}.$$

\section{Orthogonal complementarity of closed submodules associated to two projections}\label{sec:Orthogonal complementarity of closed submodules}

Throughout the rest of this paper, $A$ is a $C^*$-algebra, $H$ and $K$ are Hilbert $A$-modules. By a projection, we mean an operator $P\in \mathcal{L}(H)$ such that $P=P^2=P^*$. Recall that a closed submodule $M$ of $H$ is said to be orthogonally complemented in $H$ if $H=M\dotplus M^\perp$, where
$$M^\perp=\big\{x\in H: \langle x,y\rangle=0\ \mbox{for every }\ y\in
M\big\}.$$
In this case, the projection from $H$ onto $M$ is denoted by $P_M$.

\begin{rem}\label{rem:double orthogonality}{\rm Let $M$ be a closed submodule of $H$. Then $M$ is orthogonally complemented in $H$ if and only if there exists a projection $P\in\mathcal{L}(H)$ such that $\mathcal{R}(P)=M$. In this case,
\begin{equation}\label{equ:double orthogonality}\big(M^\bot\big)^\bot=\mathcal{R}\big(I-(I-P)\big)=\mathcal{R}(P)=M.\end{equation}
}\end{rem}

\begin{lem} \label{lem:Range closure of TT and T} {\rm\cite[Proposition 3.7]{Lance}}
Let $T\in\mathcal{L}(H,K)$. Then $$\overline {\mathcal{R}(T^*T)}=\overline{ \mathcal{R}(T^*)}\ \mbox{and}\ \overline {\mathcal{R}(TT^*)}=\overline{ \mathcal{R}(T)}.$$
\end{lem}

\begin{lem}\label{lem:R(P+Q)=R(P)+R(Q)}Let $P,Q\in\mathcal{L}(H)$ be two projections. Then
\begin{equation}\label{equ:closure of the summation of two projections}\overline{\mathcal{R}(P)+\mathcal{R}(Q)}=\overline{\mathcal{R}(P+Q)}.\end{equation}
In particular,
\begin{equation}\label{equ:range of P contained range of P plus Q}\mathcal{R}(P)\subseteq \overline{\mathcal{R}(P+Q)}\ \mbox{and}\ \mathcal{R}(Q)\subseteq \overline{\mathcal{R}(P+Q)}.
\end{equation}
\end{lem}
\begin{proof}Put
\begin{eqnarray}\label{equ:inclusion of two projectons}T=\left(
 \begin{array}{cc}
 0 & 0 \\
 P & Q \\
 \end{array}\right)\in \mathcal{L}(H\oplus H).
\end{eqnarray}
Clearly,
\begin{eqnarray}\label{equ:relationship of the ranges of two operators-1}\mathcal{R}(T)=\{0\}\oplus \big(\mathcal{R}(P)+\mathcal{R}(Q)\big), \quad \mathcal{R}(TT^*)=\{0\}\oplus \big(\mathcal{R}(P+Q)\big).
\end{eqnarray}
The equations above together with Lemma~\ref{lem:Range closure of TT and T} yield \eqref{equ:closure of the summation of two projections}, which gives
\eqref{equ:range of P contained range of P plus Q} immediately.
\end{proof}

\begin{lem}{\rm \cite[Proposition~2.5]{Liu-Luo-Xu}}\label{lem:the strict topology of the projection of the range-1} Let $T\in\mathcal{L}(H)$ be positive such that $\overline{\mathcal{R}(T)}$ is orthogonally complemented in $H$. For every $n\in \mathbb{N}$, let $T_n:=\left(\frac{1}{n}I+T\right)^{-1}T$. Then $\lim\limits_{n\to\infty}T_n=P_{\overline{\mathcal{R}(T)}}$ in the strict topology, that is,
\begin{eqnarray*} \lim_{n\to\infty}\Vert T_nx-P_{\overline{\mathcal{R}(T)}}x\Vert=0\ \mbox{for all $x$ in $H$}.\end{eqnarray*}
\end{lem}

Suppose that $P,Q\in\mathcal{L}(H)$ are two projections such that $\overline{\mathcal{R}(P+Q)}$ is orthogonally complemented in $H$. It is interesting to determine conditions under which $\overline{\mathcal{R}(I-P+I-Q)}$ is still orthogonally complemented in $H$. We provide such a result as follows.

\begin{thm}\label{The: equivalent condition of orthogonal }Let $P,Q\in\mathcal{L}(H)$ be two projections such that $\overline{\mathcal{R}(P+Q)}$ is orthogonally complemented in $H$. For every $n\in\mathbb{N}$, let $T_n=(P+Q+\frac{1}{n}I)^{-1}$ and
\begin{equation}\label{equ:defn of A_n--B_n^*}A_n=P-PT_n P,\quad B_n=PT_nQ, \quad C_n=Q-QT_n Q. \end{equation}
Then the following statements are equivalent:
\begin{enumerate}
\item[{\rm (i)}] one of $\{A_n\}, \{B_n\}$, and $\{C_n\}$ is convergent in the strict topology;
\item[{\rm (ii)}] all of $\{A_n\}, \{B_n\}$, and $\{C_n\}$ are convergent to the same limit in the strict topology;
\item[{\rm (iii)}] $\mathcal{R}(P)\cap \mathcal{R}(Q)$ is orthogonally complemented in $H$ such that
\begin{equation}\label{equ:conjecture of the orthogonal part}\big(\mathcal{R}(P)\cap \mathcal{R}(Q)\big)^\perp=\overline{\mathcal{R}(2I-P-Q)};\end{equation}
\item[{\rm (iv)}]$\overline{\mathcal{R}(2I-P-Q)}$ is orthogonally complemented in $H$.
\end{enumerate}
In each case,
\begin{eqnarray*}\lim_{n\to\infty}A_n=\lim_{n\to\infty}B_n=\lim_{n\to\infty}C_n=P_{\mathcal{R}(P)\cap \mathcal{R}(Q)}\ \mbox{in the strict topology.}\end{eqnarray*}
\end{thm}
\begin{proof} According to \eqref{equ:range of P contained range of P plus Q}, we have
\begin{equation}\label{equ:ranges P and Q contained in the closure}P_{\overline{\mathcal{R}(P+Q)}}P=PP_{\overline{\mathcal{R}(P+Q)}}=P, \quad P_{\overline{\mathcal{R}(P+Q)}}Q=QP_{\overline{\mathcal{R}(P+Q)}}=Q.
\end{equation}
Also, by Lemma~\ref{lem:the strict topology of the projection of the range-1} we have
\begin{eqnarray*}\lim\limits_{n\to\infty}(P+Q)T_n=P_{\overline{\mathcal{R}(P+Q)}}\ \mbox{in the strict topology},
\end{eqnarray*}
which is combined with \eqref{equ:ranges P and Q contained in the closure} to conclude that in the strict topology,
\begin{eqnarray}\label{eqn:the strict topology equals P} \lim\limits_{n\to\infty}\big[(P+Q)T_n P-P\big]=0, \quad \lim\limits_{n\to\infty}\big[(P+Q)T_n Q-Q\big]=0.
\end{eqnarray}
Note that for each $n\in\mathbb{N}$,
\begin{eqnarray}\label{eqn:relationship1} A_n=\big[P-(P+Q)T_nP\big]+B_n^*,\quad C_n=\big[Q-(P+Q)T_nQ\big]+B_n,
\end{eqnarray}
and
\begin{eqnarray}
\label{eqn:relationship2} B_n=\big[(P+Q)T_n Q-Q\big]+\big[Q-QT_n(P+Q)\big]+B_n^*.
\end{eqnarray}
Hence the equivalence of (i) and (ii) can be derived from the equations above together with \eqref{eqn:the strict topology equals P}, \eqref{eqn:relationship1}, and \eqref{eqn:relationship2}.
Furthermore, if any of $A_n, B_n$ and $C_n$ has the limit in the strict topology, then all of them will have the same limit in the strict topology.

(ii)$\Longrightarrow$(iii): Suppose that
\begin{eqnarray}\label{limts A_n=B_n=C_n=B_n^*}\lim\limits_{n\to\infty}A_n=\lim\limits_{n\to\infty}B_n=\lim\limits_{n\to\infty}C_n=E\ \mbox{in the strict topology}.
\end{eqnarray} From the definitions of $B_n$ and $C_n$ in \eqref{equ:defn of A_n--B_n^*}, we have $\mathcal{R}(B_n)\subseteq \mathcal{R}(P)$ and $\mathcal{R}(C_n)\subseteq \mathcal{R}(Q)$ for each $n$. Employing \eqref{limts A_n=B_n=C_n=B_n^*} we get $\mathcal{R}(E)\subseteq \mathcal{R}(P) \cap \mathcal{R}(Q)$.
Conversely, given each $x\in \mathcal{R}(P) \cap \mathcal{R}(Q)$, we have
$$\Big(P+Q+\frac1n I\Big)x=\Big(2+\frac1n\Big)x\ \mbox{for every $n\in\mathbb{N}$},$$
which means that
$T_n x=\frac{1}{2+\frac1n}x$ and thus $B_nx=\frac{1}{2+\frac1n}x$. Hence
\begin{equation}\label{equ:2E is an idempotent}2Ex=2\lim\limits_{n\to\infty} B_n x=x.\end{equation} This ensures that $\mathcal{R}(E)=\mathcal{R}(P) \cap \mathcal{R}(Q)$.

Now for every $x\in H$, since $Ex\in \mathcal{R}(P) \cap \mathcal{R}(Q)$, from
\eqref{equ:2E is an idempotent} we can get
$$(2E)^2x=(2E)x.$$ Hence the operator $2E$ is an idempotent. Moreover, as
$$2E=\lim\limits_{n\to\infty} (2A_n)\ \mbox{in the strict topology and}\ A_n=A_n^*\ \mbox{for every $n\in\mathbb{N}$},$$ we see that $2E$ is also self-adjoint. Therefore,
$2E$ is actually a projection. It follows that $\mathcal{R}(P) \cap \mathcal{R}(Q)$ is orthogonally complemented in $H$ and
\begin{equation}\label{equ:orthononal complement-1}\big(\mathcal{R}(P) \cap \mathcal{R}(Q)\big)^\perp=\mathcal{R}(I-2E).\end{equation}

Next, we prove that \eqref{equ:conjecture of the orthogonal part} is valid. It is obvious that
$$\overline{\mathcal{R}(2I-P-Q)}=\overline{\mathcal{R}(I-P+I-Q)}\subseteq \big(\mathcal{R}(P)\cap \mathcal{R}(Q)\big)^\perp.$$
On the other hand, we have
$\lim\limits_{n\to\infty}B_n^*=E^*=E$ in the strict topology. Therefore, given any $x\in H$, we have
\begin{eqnarray*}(I-2E)x&=&\lim\limits_{n\to\infty}(I-A_n-B_n^*)x \nonumber\\
&=&\lim\limits_{n\to\infty} \big[(I-P)x+(P-Q)T_nPx\big] \nonumber\\
&=&\lim\limits_{n\to\infty} \big[(I-P)(x-T_n Px)+(I-Q)T_nPx\big]\nonumber\\
&\in& \overline{\mathcal{R}(I-P)+\mathcal{R}(I-Q)}.
\end{eqnarray*}
Due to \eqref{equ:orthononal complement-1} and \eqref{equ:closure of the summation of two projections}, we have
$$\big(\mathcal{R}(P)\cap \mathcal{R}(Q)\big)^\perp\subseteq
\overline{\mathcal{R}(I-P)+\mathcal{R}(I-Q)}=\overline{\mathcal{R}(2I-P-Q)}.$$
The proof of \eqref{equ:conjecture of the orthogonal part} is then finished.

(iii)$\Longrightarrow$(iv) It is illustrated by Remark~\ref{rem:double orthogonality}.

(iv)$\Longrightarrow$(i): Assume that $\overline{\mathcal{R}(2I-P-Q)}$ is orthogonally complemented in $H$.
Since $\overline{\mathcal{R}(2I-P-Q)}^\bot=\mathcal{N}(I-P)\cap\mathcal{N}(I-Q)=\mathcal{R}(P)\cap \mathcal{R}(Q)$, by \eqref{equ:double orthogonality}
we know that $\mathcal{R}(P)\cap \mathcal{R}(Q)$ is also orthogonally complemented in $H$ such that \eqref{equ:conjecture of the orthogonal part} is satisfied.
So the notation $P_{\mathcal{R}(P)\cap \mathcal{R}(Q)}$ is meaningful, and $H$ can be decomposed orthogonally as
\begin{equation}\label{equ:new decomposition of H}H=\overline{\mathcal{R}(2I-P-Q)}\dotplus \mathcal{R}(P)\cap \mathcal{R}(Q).\end{equation}
In what follows, we prove that
\begin{equation}\label{equ:limit value of B n}\lim\limits_{n\to\infty} 2B_n =P_{\mathcal{R}(P)\cap \mathcal{R}(Q)}\ \mbox{in the strict topology}.\end{equation}

First, given any $x\in \mathcal{R}(2I-P-Q)$, there exists some $u\in H$ such that $x=(2I-P-Q)u=(I-P+I-Q)u$. Then
\begin{align}\lim\limits_{n\to\infty} B_n x&=\lim\limits_{n\to\infty} B_n (I-P)u=\lim\limits_{n\to\infty} \big[PT_n(P+Q)-PT_nP\big](I-P)u\nonumber\\
\label{eqn:tends to zero-1} &=\lim\limits_{n\to\infty} PT_n(P+Q)(I-P)u=P(I-P)u=0.
\end{align}
Note that for every $n\in\mathbb{N}$, we have
\begin{eqnarray*}\Vert T_n^\frac12 P\Vert^2=\Vert T_n^\frac12 P\cdot (T_n^\frac12 P)^*\Vert=\Vert T_n^\frac12 P T_n^\frac12\Vert\le \Vert T_n^\frac12 (P+Q)T_n^\frac12\Vert<1,
\end{eqnarray*}
so $\Vert PT_n^\frac12\Vert=\Vert (PT_n^\frac12)^*\Vert=\Vert T_n^\frac12 P\Vert<1$. Similarly, $\Vert T_n^\frac12 Q\Vert<1$.
As a result,
\begin{eqnarray*}\label{equ:norm of B_n is bounded}\Vert B_n\Vert\le \Vert PT_n^\frac12\Vert\cdot \Vert T_n^\frac12 Q\Vert<1,\ \mbox{for every $n\in\mathbb{N}$}.\end{eqnarray*}
The boundedness of $\{B_n\}$ together with \eqref{eqn:tends to zero-1} indicates
\begin{equation} \label{equ:tends to zero-2}\lim\limits_{n\to\infty} 2B_n x=0\ \mbox{for every $x\in \overline{\mathcal{R}(2I-P-Q)}$}.\end{equation}

Next, from the proof of (ii)$\Longrightarrow$(iii) we know that
\begin{equation}\label{equ:tends to half value}\lim\limits_{n\to\infty} 2B_n x=x, \ \mbox{for every $x\in \mathcal{R}(P)\cap\mathcal{P}(Q)$}.\end{equation}
The assertion $\lim\limits_{n\to\infty} 2B_nx =P_{\mathcal{R}(P)\cap \mathcal{R}(Q)}x\,\,(x\in H)$ follows from \eqref{equ:new decomposition of H}, \eqref{equ:tends to zero-2} and \eqref{equ:tends to half value}. Similarly, one can show that $\lim\limits_{n\to\infty} 2B_n^*x =P_{\mathcal{R}(P)\cap \mathcal{R}(Q)}x\,\,(x\in H)$. Thus \eqref{equ:limit value of B n} holds true.
\end{proof}

\begin{rem}{\rm It is remarkable that
there exist a Hilbert $C^*$-module $H$ and two projections $P,Q\in\mathcal{L}(H)$ such that $\mathcal{R}(P)\cap \mathcal{R}(Q)$ is orthogonally complemented in $H$, whereas \eqref{equ:conjecture of the orthogonal part} is not true. Such an example is constructed in the next section.
}\end{rem}

\section{An example of extremely discomplementable projections}\label{sec:discomplementable projections}
 Inspired by \cite[Section~3]{Manuilov-Moslehian-Xu}, we construct a Hilbert $C^*$-module $H$ and two extremely discomplementable projections on it as follows.

Let $M_2(\mathbb{C})$ be the set of all $2\times 2$ complex matrices and $\|\cdot\|$ be the spectral norm on $M_2(\mathbb{C})$.
Let $A=C\big([0,1];M_2(\mathbb{C})\big)$ be the set of all continuous matrix-valued functions from $[0,1]$ to $M_2(\mathbb{C})$. Set
\begin{align*}x^*(t)=\big(x(t)\big)^*\ \mbox{and}\ \|x\|=\max_{0\le s\le 1}\|x(s)\|,\ \mbox{for every $x\in A$ and each $t\in [0,1]$.}\end{align*}
With the $*$-operation above together with the usual algebraic operations, $A$ is a unital $C^*$-algebra.
Therefore, $A$ itself becomes a Hilbert $A$-module with the usual $A$-valued inner product given by
$$\langle x,y\rangle=x^*y,\ \mbox{for every $x,y\in A$}.$$

Let $e$ be the unit of $A$, that is, $e(t)=\left(
 \begin{array}{cc}
 1 & 0 \\
 0 & 1 \\
 \end{array}
 \right)$ for every $t\in [0,1]$.
It is known that $A\cong \mathcal{L}(A)$ via $a\to L_a$, where
$L_a(x)=ax\ \mbox{for every $a,x\in A$}$; furthermore $(L_a)^*=L_{a^*}$. Indeed,
for every $T\in\mathcal{L}(A)$ and $x\in A$, we have
$$Tx=T(ex)=T(e)x=L_ax,\ \mbox{where $a=T(e)$.}$$

For simplicity, we put
$$c_t=\cos\frac{\pi}{2}t\ \mbox{and}\ s_t=\sin\frac{\pi}{2}t,\ \mbox{for each $t\in[0,1]$}.$$ Let
$\widetilde{P},\widetilde{Q}\in A$ be determined by the matrix-valued functions
\begin{eqnarray*}\label{equ:defn of P t and Q t}\widetilde{P}(t)\equiv\left(\begin{matrix}1&0\\0&0\end{matrix}\right)\ \mbox{and}\ \widetilde{Q}(t)=\left(\begin{matrix}c_t^2&s_tc_t\\s_tc_t&s_t^2\end{matrix}\right), \ \mbox{for each $t\in [0,1]$}.\end{eqnarray*}
Then both $P=L_{\widetilde{P}}$ and $Q=L_{\widetilde{Q}}$ are projections in $\mathcal{L}(A)$. Let $x\in A$ be determined by $x(t)=\big(x_{ij}(t)\big)_{1\le i,j\le 2}$, where each $x_{ij}(t)$ is a continuous complex-valued function on $[0,1]$. If $x\in \mathcal{N}(P)$, then obviously $x_{11}(t)= 0$ and $x_{12}(t)= 0$. So if furthermore $x\in\mathcal{N}(Q)$, then
$$s_t^2\, x_{21}(t)=0\ \mbox{and}\ s_t^2\, x_{22}(t)=0\ \mbox{for each $t\in [0,1]$}.$$
Note that $s_t\ne 0$ for every $t\in (0,1]$, so the equations above together with the continuity of both $x_{21}(t)$ and $x_{22}(t)$ at $t=0$ yield
 $x_{21}(t)= 0$ and $x_{22}(t)= 0$. Therefore, $\mathcal{N}(P)\cap \mathcal{N}(Q)=\{0\}$; or equivalently,
$\mathcal{N}(P+Q)=\{0\}$.

Consider an element $x\in A$ having the form $x(t)=\big(x_{ij}(t)\big)_{1\le i,j\le 2}$. Since $s_0=0$, we have
\begin{align*}\|(P+Q)\,x-e\|=\| \widetilde{P}x+\widetilde{Q}x-e\|\ge\max_{0\le t\le 1}\big|s_tc_t x_{12}(t)+s_t^2 x_{22}(t)-1\big|\ge 1.
\end{align*}
 It follows immediately that $e\notin \overline{\mathcal{R}(P+Q)}$, whence
$$\mathcal{N}(P+Q)\dotplus\overline{\mathcal{R}(P+Q)}=\overline{\mathcal{R}(P+Q)}\ne A.$$
Therefore, $\overline{\mathcal{R}(P+Q)}$ is not orthogonally complemented.

Similarly, it can be proved that
$$\mathcal{N}(P+I-Q)=\mathcal{N}(I-P+Q)=\mathcal{N}(I-P+I-Q)=\{0\},$$
and the unit $e$ is also not contained in any one of the remaining three closures in \eqref{equ:four closures}.

\section{Halmos' two projections theorem for Hilbert $C^*$-module operators}\label{sec:Halmos' two projections theorem}
The purpose of this section is to generalize Halmos' two projections theorem to the case of the Hilbert $C^*$-module.
We start this section with the following lemma.

\begin{lem}\label{lem:derivation of the projection-1} Let $P,Q\in\mathcal{L}(H)$ be two projections such that $\overline{\mathcal{R}(I-Q+P)}$ is
orthogonally complemented in $H$. Then $\overline{\mathcal{R}(QP)}$ is also orthogonally complemented in $H$
and
\begin{equation}\label{equ:formula for the projection on QP}P_{\overline{\mathcal{R}(QP)}}=Q-P_{\mathcal{R}(Q)\cap \mathcal{N}(P)}.
\end{equation}
\end{lem}
\begin{proof} First, we prove that
\begin{equation}\label{equ:range Q=range QP}\mathcal{R}(Q)=\overline{\mathcal{R}(QP)}+\mathcal{R}(Q)\cap \mathcal{N}(P).\end{equation}
Indeed, it is clear that $$\mathcal{N}(I-Q+P)=\mathcal{N}(I-Q)\cap \mathcal{N}(P)=\mathcal{R}(Q)\cap \mathcal{N}(P),$$
which leads to the orthogonal decomposition of $H$ as
\begin{eqnarray*}\label{equ:decomposition of H-11}H=\overline{\mathcal{R}(I-Q+P)}\dotplus \mathcal{R}(Q)\cap \mathcal{N}(P),
\end{eqnarray*}
since $\overline{\mathcal{R}(I-Q+P)}$ is orthogonally complemented. As a result, the notation of $P_{\mathcal{R}(Q)\cap \mathcal{N}(P)}$
is meaningful.

Now, given any $x\in H$, $x$ can be decomposed as
\begin{eqnarray*}\label{equ:x fenjie}x=u+v\ \mbox{for some $u\in \overline{\mathcal{R}(I-Q+P)}$ and $v\in \mathcal{R}(Q)\cap \mathcal{N}(P)$}.
\end{eqnarray*}
Thus
\begin{equation}\label{equ:Qx}Qx=Qu+Qv=Qu+v\in Qu+\mathcal{R}(Q)\cap \mathcal{N}(P).
\end{equation}
As $u\in \overline{\mathcal{R}(I-Q+P)}$, there exists a sequence $\{x_n\}\subseteq H$ such that $(I-Q+P)x_n\rightarrow u.$ Then $QPx_n=Q(I-Q+P)x_n\rightarrow Qu$,
and hence $Qu\in \overline{\mathcal{R}(QP)}$.
It follows from \eqref{equ:Qx} that
\begin{eqnarray*}\label{Qx subset}Qx\in \overline{\mathcal{R}(QP)}+\mathcal{R}(Q)\cap \mathcal{N}(P),
\end{eqnarray*}
which gives \eqref{equ:range Q=range QP} since $x\in H$ is arbitrary and $\mathcal{R}(QP)\subseteq \mathcal{R}(Q)$.

Next, we prove that
\begin{equation}\label{equ:check for the orthogonality-1}\overline{\mathcal{R}(QP)}\bot \mathcal{R}(Q)\cap \mathcal{N}(P).\end{equation}
In fact, given any $s\in H$ and $t\in \mathcal{R}(Q)\cap \mathcal{N}(P)$, we have
\begin{eqnarray*}\label{inner product of t and QPs}\langle QPs,t\rangle=\langle s,PQt\rangle=\langle s,Pt\rangle=\langle s,0\rangle=0.
\end{eqnarray*}
The proof of \eqref{equ:check for the orthogonality-1} is then finished by the continuity of the $A$-valued inner product with respect to each variable.

Finally, it follows from \eqref{equ:range Q=range QP} and \eqref{equ:check for the orthogonality-1} that
\begin{eqnarray*}\label{equ:range Q=range QP-orthogonal-1}\mathcal{R}(Q)=\overline{\mathcal{R}(QP)}\dotplus \mathcal{R}(Q)\cap \mathcal{N}(P).
\end{eqnarray*}
Hence \eqref{equ:formula for the projection on QP} is satisfied.
\end{proof}

\begin{lem}\label{lem:derivation of the projection-2} Let $P,Q\in\mathcal{L}(H)$ be two projections such that $\overline{\mathcal{R}(2I-P-Q)}$ is
orthogonally complemented in $H$. Then $\mathcal{R}(P)\cap \mathcal{R}(Q)$ is also orthogonally complemented in $H$ and an orthogonal decomposition of
$\overline{\mathcal{R}(QP)}$ can be given as
\begin{equation}\label{equorthogonal decomposition of the range of QP}\overline{\mathcal{R}(QP)}= \overline{\mathcal{R}\big(QP(I-Q)\big)}\dotplus \mathcal{R}(P)\cap \mathcal{R}(Q).\end{equation}
\end{lem}
\begin{proof} Since $\overline{\mathcal{R}(2I-P-Q)}$ is orthogonally complemented in $H$, the notation $P_{\mathcal{R}(P)\cap \mathcal{R}(Q)}$ is meaningful, and $H$ can be decomposed orthogonally as \eqref{equ:new decomposition of H}.
Replacing $Q$ and $P$ in Lemma~\ref{lem:derivation of the projection-1} with $P$ and $I-Q$ respectively, by \eqref{equ:formula for the projection on QP} we obtain
\begin{eqnarray*}\mathcal{R}(P)=\overline{\mathcal{R}\big(P(I-Q)\big)}\dotplus \mathcal{R}(P)\cap \mathcal{R}(Q),
\end{eqnarray*}
which clearly gives
\begin{eqnarray*}\mathcal{R}(QP)&\subseteq& Q\overline{\mathcal{R}\big(P(I-Q)\big)}+ \mathcal{R}(P)\cap \mathcal{R}(Q)\\
&\subseteq& \overline{\mathcal{R}\big(QP(I-Q)\big)}+\mathcal{R}(P)\cap \mathcal{R}(Q).
\end{eqnarray*}
Note that $\overline{\mathcal{R}\big(QP(I-Q)\big)}$ and $\mathcal{R}(P)\cap \mathcal{R}(Q)$ are orthogonal to each other, so their sum is closed. Hence
\begin{eqnarray*}\overline{\mathcal{R}(QP)}\subseteq \overline{\mathcal{R}\big(QP(I-Q)\big)}+\mathcal{R}(P)\cap \mathcal{R}(Q)\subseteq \overline{\mathcal{R}(QP)},\end{eqnarray*}
since $QPx=x$ for every $x\in\mathcal{R}(P)\cap \mathcal{R}(Q)$.
The proof of \eqref{equorthogonal decomposition of the range of QP} is then finished.
\end{proof}

A direct application of Lemmas~\ref{lem:derivation of the projection-1} and \ref{lem:derivation of the projection-2} is as follows.

\begin{cor}\label{cor: a projection derived two orthogonalities} Let $P,Q\in\mathcal{L}(H)$ be two projections. If both $\overline{\mathcal{R}(I-Q+P)}$ and $\overline{\mathcal{R}(2I-P-Q)}$ are orthogonally complemented in $H$, then
$\overline{\mathcal{R}\big(QP(I-Q)\big)}$ is also orthogonally complemented in $H$ and
\begin{equation}\label{make projection for}P_{\overline{\mathcal{R}\big(QP(I-Q)\big)}}=Q-P_{\mathcal{R}(Q)\cap \mathcal{N}(P)}-P_{\mathcal{R}(Q)\cap \mathcal{R}(P)}.
\end{equation}
\end{cor}

\begin{defn}\label{defn:harmonious projections} Two projections $P,Q\in\mathcal{L}(H)$ are said to be harmonious if the four closures in \eqref{equ:four closures} are all orthogonally complemented in $H$.
\end{defn}

Suppose that $P,Q\in\mathcal{L}(H)$ are two harmonious projections. Let
\begin{eqnarray}\label{eqn:defn of H1 and H2}&&H_1=\mathcal{R}(P)\cap\mathcal{R}(Q), \ H_2=\mathcal{R}(P)\cap\mathcal{N}(Q),\\
 \label{eqn:defn of H3 and H4}&&H_3=\mathcal{N}(P)\cap\mathcal{R}(Q), \ H_4=\mathcal{N}(P)\cap\mathcal{N}(Q).
\end{eqnarray}
Since $\overline{\mathcal{R}(P+Q)}^\bot=H_4$ and $\overline{\mathcal{R}(P+Q)}$ is orthogonally complemented in $H$,
we conclude from \eqref{equ:double orthogonality} that $H_4$ is also orthogonally complemented in $H$.
Similarly, $H_1,H_2$ and $H_3$ are all orthogonally complemented in $H$. Let
\begin{equation}\label{equ:defn of P1--P4}P_{H_i}\ \mbox{be denoted simply by $P_i$ for $i=1,2,3,4$},\end{equation}
and put
\begin{align}\label{eqn:defn of P5}&P_5=P-P_1-P_2\ \mbox{and}\ H_5=\mathcal{R}(P_5),\\
\label{eqn:defn of P6}&P_6=(I-P)-P_3-P_4\ \mbox{and}\ H_6=\mathcal{R}(P_6).
\end{align}
With the notations given above, a unitary operator $U_{P,Q}: H\to \oplus_{i=1}^6 H_i$ is given by
\begin{equation}\label{equ:defn of V}U_{P,Q}(x)=\Big(P_1(x),P_2(x),\cdots,P_6(x)\Big)^T\ \mbox{for every}\ x\in H,\end{equation}
with the property that
$$U_{P,Q}^*\Big((x_1,x_2,\cdots,x_6)^T\Big)=\sum_{i=1}^6 x_i,\ \mbox{for every $x_i\in H_i$}, \quad i=1,2,\cdots,6.$$ It follows that
\begin{eqnarray}\label{equ:decomposition of P}U_{P,Q}\,P\,U_{P,Q}^*&=&I_{H_1}\oplus I_{H_2}\oplus 0\oplus 0\oplus I_{H_5}\oplus 0,\\
\label{equ:decomposition of Q}U_{P,Q}\,Q\,U_{P,Q}^*&=&I_{H_1}\oplus 0\oplus I_{H_3}\oplus 0\oplus T,
\end{eqnarray}
where
\begin{equation}\label{equ:2 by 2 blocked operator matrix for Q-1}
T=\left(\begin{array}{cc}P_5QP_5|_{H_5}&P_5QP_6|_{H_6}\\(P_5QP_6|_{H_6})^*&P_6QP_6|_{H_6}\end{array}\right)\in\mathcal{L}(H_5\oplus H_6),
\end{equation}
in which $P_5QP_5|_{H_5}$ is the restriction of the operator $P_5QP_5$ on $H_5$. The same convention is taken for
$P_5QP_6|_{H_6}$ and $P_6QP_6|_{H_6}$.

\begin{lem}\label{lem:prepare for the polar decomposition-1} Suppose that $P,Q\in\mathcal{L}(H)$ are two harmonious projections. Let $H_i, P_i$ $(1\le i\le 6$) be defined by \eqref{eqn:defn of H1 and H2}--\eqref{eqn:defn of P6}, respectively. Then
\begin{align}\label{eqn:prepare for the polar decomposition-1}&\overline{\mathcal{R}(P_5Q)}=H_5, \quad \overline{\mathcal{R}(QP_5)}= \overline{\mathcal{R}\big(QP(I-Q)\big)}, \\
\label{eqn:prepare for the polar decomposition-2}&\overline{\mathcal{R}(P_6Q)}=H_6, \quad \overline{\mathcal{R}(QP_6)}= \overline{\mathcal{R}\big(QP(I-Q)\big)}.
\end{align}
\end{lem}
\begin{proof}Exchanging $P$ with $Q$, we observe from \eqref{make projection for} and \eqref{eqn:defn of P5} that
\begin{equation}\label{equ:middle characterization of projection-1}P_{\overline{\mathcal{R}\big(PQ(I-P)\big)}}=P-P_{\mathcal{R}(P)\cap \mathcal{N}(Q)}-P_{\mathcal{R}(P)\cap \mathcal{R}(Q)}=P-P_2-P_1=P_5.
\end{equation}
Moreover, we have $(I-P_1)(I-P)=I-P$ and
\begin{eqnarray*}P_5Q&=&(P-P_1-P_2)Q=PQ-P_1=PQ(I-P_1),
\end{eqnarray*}
which shows that
$$\mathcal{R}\big(PQ(I-P)\big)\subseteq \mathcal{R}\big(PQ(I-P_1)\big)=\mathcal{R}(P_5Q)\subseteq \mathcal{R}(P_5)=H_5.$$
The inclusions of the above sets together with \eqref{equ:middle characterization of projection-1} yield
$\overline{\mathcal{R}(P_5Q)}=H_5$.

Note that $I-Q=(I-P_1)(I-Q)$ and
\begin{eqnarray*}QP_5=Q(P-P_1-P_2)=QP-P_1=QP(I-P_1),
\end{eqnarray*}
so
$$\overline{\mathcal{R}\big(QP(I-Q)\big)}\subseteq \overline{\mathcal{R}\big(QP(I-P_1)\big)}=\overline{\mathcal{R}(QP_5)}.$$
Meanwhile, by \eqref{make projection for}, we have
$$QP_5=P_{\overline{\mathcal{R}\big(QP(I-Q)\big)}}P_5+P_3P_5+P_1P_5=P_{\overline{\mathcal{R}\big(QP(I-Q)\big)}}P_5,
$$
so $\overline{\mathcal{R}(QP_5)}\subseteq \overline{\mathcal{R}\big(QP(I-Q)\big)}$. Therefore $\overline{\mathcal{R}(QP_5)}= \overline{\mathcal{R}\big(QP(I-Q)\big)}$. This completes the proof of \eqref{eqn:prepare for the polar decomposition-1}.

Replacing the pair $(P,Q)$ with $(I-P,Q)$, we have
\begin{eqnarray*}\widetilde{H}_1=\mathcal{R}(I-P)\cap\mathcal{R}(Q)=H_3, \quad\widetilde{P}_1=P_{\widetilde{H}_1}=P_3.
\end{eqnarray*}
Similarly,
\begin{align*}\widetilde{H}_2=H_4,\quad \widetilde{H}_3=H_1,\quad \widetilde{H}_4=H_2,\quad \widetilde{H}_5=H_6\ \mbox{and}\ \widetilde{H}_6=H_5.
\end{align*}
Note that $Q(I-P)(I-Q)=-QP(I-Q)$, so from \eqref{eqn:prepare for the polar decomposition-1} we obtain
\begin{align*}\overline{\mathcal{R}(P_6Q)}&=\overline{\mathcal{R}(\widetilde{P}_5Q)}=\widetilde{H}_5=H_6,\\
\overline{\mathcal{R}(QP_6))}&=\overline{\mathcal{R}\big(Q\widetilde{P}_5)\big)}=\overline{\mathcal{R}\big(Q(I-P)(I-Q)\big)}=\overline{\mathcal{R}\big(QP(I-Q)\big)}. \qedhere
\end{align*}

\end{proof}

\begin{defn}\cite{Manuilov-Moslehian-Xu} An operator $T\in\mathcal{L}(H,K)$ is said to be semi-regular if $\overline{\mathcal{R}(T)}$ and $\overline{\mathcal{R}(T^*)}$ are orthogonally complemented in $K$ and $H$, respectively.
\end{defn}

\begin{lem}\label{lem:Wegge-Olsen}{\rm (\cite[Lemma~3.6]{Liu-Luo-Xu} and \cite[Proposition~15.3.7]{Wegge-Olsen})}\ Let $T\in\mathcal{L}(H,K)$ be semi-regular. Then there exists a partial isometry $U\in\mathcal{L}(H,K)$ such that
\begin{eqnarray*}\label{equ:equations associated the polar decomposition}T=U(T^*T)^{\frac{1}{2}}=(TT^*)^\frac12 U,\quad U^*U=P_{\overline{\mathcal{R}(T^*)}}, \quad\ UU^*=P_{\overline{\mathcal{R}(T)}}.
\end{eqnarray*}
\end{lem}

Halmos' two projections theorem for Hilbert space operators has several
equivalent versions \cite{Bottcher-Spitkovsky}, one of which turns out to be \cite[Theorem~1.4]{Deng-Du}. It can be generalized for Hilbert $C^*$-module operators as follows.

\begin{thm}\label{thm:Halmos' two projections theorem}{\rm (cf.\,\cite[Theorem~1.4]{Deng-Du})} Suppose that $P,Q\in\mathcal{L}(H)$ are two harmonious projections. Let $H_i, P_i$ $(1\le i\le 6$) be defined by \eqref{eqn:defn of H1 and H2}--\eqref{eqn:defn of P6}, respectively. Then the operator $T$ given by \eqref{equ:2 by 2 blocked operator matrix for Q-1} can be characterized as
\begin{equation}\label{equ:2 by 2 blocked operator matrix for Q-2}T=\left(\begin{array}{cc}Q_0&Q_0^\frac{1}{2}(I_{H_5}-Q_0)^\frac{1}{2}U_0\\U_0^* Q_0^\frac{1}{2}(I_{H_5}-Q_0)^\frac{1}{2}&U_0^*(I_{H_5}-Q_0)U_0\end{array}\right)\in\mathcal{L}(H_5\oplus H_6),\end{equation}
where
 $U_0\in\mathcal{L}(H_6,H_5)$ is a unitary operator, $Q_0$ is the restriction of $P_5QP_5$ on $H_5$, and both $Q_0$ and $I_{H_5}-Q_0$ are positive, injective and contractive.
\end{thm}
\begin{proof}By Lemma~\ref{lem:prepare for the polar decomposition-1} and \eqref{make projection for},
 all $\overline{\mathcal{R}(P_5Q)}$, $\overline{\mathcal{R}(QP_5)}$, $\overline{\mathcal{R}(P_6Q)}$ and $\overline{\mathcal{R}(QP_6)}$ are orthogonally complemented in $H$. Therefore, by Lemma~\ref{lem:Wegge-Olsen}, there exist partial isometries $U, V\in\mathcal{L}(H)$ such that
\begin{align}\label{eqn:polar decompositions fro two operators}P_5Q&=(P_5QP_5)^{\frac{1}{2}} U, \ QP_6=V(P_6QP_6)^{\frac{1}{2}},\\
\label{eqn:U star is an isometry}UU^*&=P_{\overline{\mathcal{R}(P_5Q)}}=P_5=I_{H_5}\ \mbox{by \eqref{eqn:prepare for the polar decomposition-1}},\\
 \label{eqn:V is an isometry}V^*V&=P_{\overline{\mathcal{R}(P_6Q)}}=P_6=I_{H_6} \ \mbox{by \eqref{eqn:prepare for the polar decomposition-2}},\\
\label{eqn:half equality of U and V}U^*U&=P_{\overline{\mathcal{R}(QP_5)}}=P_{\overline{\mathcal{R}(QP_6)}}=VV^* \ \mbox{by \eqref{eqn:prepare for the polar decomposition-1} and \eqref{eqn:prepare for the polar decomposition-2}}.
\end{align}

Let $Q_0\in\mathcal{L}(H_5)$ and $Q_1\in\mathcal{L}(H_6)$ be defined respectively by
\begin{equation}\label{equ:defn of Q0 and Q1}Q_0=P_5QP_5|_{H_5}, \quad \ Q_1=P_6QP_6|_{H_6}.\end{equation}
Clearly, $Q_0$ is positive and contractive. Furthermore, we have
\begin{equation}\label{equ:injectivity of Q0 and I-Q0}\mathcal{N}(Q_0)=\mathcal{N}(I_{H_5}-Q_0)=\{0\}.\end{equation}
In fact, if $x=P_5x\in H_5$ is given such that $Q_0x=0$, then $Qx=QP_5x=0$, which means that $x\in H_5\cap \mathcal{N}(Q)=H_5\cap \mathcal{R}(P)\cap\mathcal{N}(Q)=H_5\cap H_2=\{0\}$.
Similarly, if $y\in H_5$ is such that $(I_{H_5}-Q_0)y=0$, then $P_5(I-Q)P_5y=0$, and hence $(I-Q)y=(I-Q)P_5y=0$, which gives
$y\in H_5\cap H_1=\{0\}$. This completes the proof of \eqref{equ:injectivity of Q0 and I-Q0}.

Let $U_0=UV|_{H_6}$ be the restriction of $UV$ on $H_6$. Then by \eqref{eqn:U star is an isometry} and \eqref{eqn:V is an isometry}, we see that $U_0\in\mathcal{L}(H_6,H_5)$ with $U_0^*=V^*U^*|_{H_5}$.
 We prove that $U_0$ is a unitary. Indeed, for every $x\in H_6$, by \eqref{eqn:V is an isometry} and \eqref{eqn:half equality of U and V} we have
$$U_0^* U_0x=V^*U^*UVx=(V^*V)^2x=x,$$
which means that $U_0^*U_0=I_{H_6}$. Similarly, for every $y\in H_5$, by \eqref{eqn:U star is an isometry} and \eqref{eqn:half equality of U and V} we have
 $$U_0U_0^*y=UVV^*U^*y=(UU^*)^2y=y,$$
whence $U_0U_0^*=I_{H_5}$. Thus $U_0\in\mathcal{L}(H_6,H_5)$ is a unitary.

 Next, we prove that
 \begin{equation}\label{equ:relationship between Q0 and Q1}Q_1=U_0^*(I_{H_5}-Q_0)U_0,\end{equation} where $Q_0$ and $Q_1$ are defined by \eqref{equ:defn of Q0 and Q1}.
In fact, from \eqref{eqn:polar decompositions fro two operators} and \eqref{eqn:U star is an isometry} we have
$$P_5QP_6|_{H_6}=Q_0^\frac12 UVQ_1^\frac12=Q_0^\frac12 U_0Q_1^\frac12,$$
therefore the operator $T$ defined by \eqref{equ:2 by 2 blocked operator matrix for Q-1} can be expressed alternately as
\begin{equation}\label{equ:an alternative expression for T}
T=\left(\begin{array}{cc}Q_0&Q_0^{\frac{1}{2}}U_0 Q_1^{\frac{1}{2}}\\Q_1^{\frac{1}{2}}U_0^* Q_0^{\frac{1}{2}}&Q_1\end{array}\right),
\end{equation}
which, in virtue of $T^2=T$, gives
\begin{eqnarray*}\label{equ:entry 1 1 is the same}Q_0^2+Q_0^{\frac{1}{2}}U_0Q_1U_0^* Q_0^{\frac{1}{2}}=Q_0.\end{eqnarray*}
The above equation together with \eqref{equ:injectivity of Q0 and I-Q0} yields
\begin{eqnarray*}Q_0^\frac32+U_0Q_1U_0^* Q_0^{\frac{1}{2}}=Q_0^\frac12.\end{eqnarray*}
Taking $*$-operation, we get
\begin{eqnarray*}Q_0^\frac32+Q_0^{\frac{1}{2}} U_0Q_1U_0^* =Q_0^\frac12.\end{eqnarray*}
Once again, by the injectivity of $Q_0^\frac12$, we can obtain $Q_0+U_0Q_1U_0^* =I_{H_5}$, which clearly leads to
\eqref{equ:relationship between Q0 and Q1}. Since $U_0$ is a unitary, from
\eqref{equ:relationship between Q0 and Q1} we can obtain
\begin{equation}\label{equ:relationship between Q0 and Q1-2}Q_1^\frac12=U_0^*(I_{H_5}-Q_0)^\frac12U_0.\end{equation}
Formula \eqref{equ:2 by 2 blocked operator matrix for Q-2} for $T$ then follows from \eqref{equ:relationship between Q0 and Q1}--\eqref{equ:relationship between Q0 and Q1-2}.
\end{proof}

\section{The norm equation concerning the characterization of the Friedrichs angle}\label{sec:the Friedrichs angle}

In this section, we focus on the study of the validity of equation~\eqref{equ:motivation equation}. First, we give a partial positive answer as follows.

\begin{lem}\label{lem:application of Halmos two projections-1} Equation \eqref{equ:motivation equation} is true for every two harmonious projections $P$ and $Q$.
\end{lem}
\begin{proof} Let $P,Q\in\mathcal{L}(H)$ be two harmonious projections. It needs only to prove that
\begin{equation}\label{equ:norm equivalent for the angle-1}\Vert P QP_{\overline{\mathcal{R}(2I-P-Q)}}\Vert=\Vert (I-P)(I-Q)P_{\overline{\mathcal{R}(P+Q)}}\Vert.\end{equation}
Let $H_i, P_i$ $(1\le i\le 6$) be defined by \eqref{eqn:defn of H1 and H2}--\eqref{eqn:defn of P6}, respectively. Denote $I_{H_i}$ simply by $I$ for $1\le i\le 6$.
Let $U_{P,Q}$ be defined by \eqref{equ:defn of V} and
for each $n\in\mathbb{N}$, put
\begin{align}\label{eqn:defn of angle A n}&A_n=U_{P,Q}PQ(2I-P-Q)\Big[2I-P-Q+\frac{I}{n}\Big]^{-1}U_{P,Q}^*,\\
\label{eqn:defn of angle B n}&B_n=U_{P,Q}(I-P)(I-Q)(P+Q)\Big[P+Q+\frac{I}{n}\Big]^{-1}U_{P,Q}^*.
\end{align}
 It can be deduced directly from \eqref{equ:decomposition of P}, \eqref{equ:decomposition of Q} and \eqref{equ:2 by 2 blocked operator matrix for Q-2} that
\begin{eqnarray*}\label{equ:introduce M1}&&U_{P,Q}(PQ)U_{P,Q}^*=I\oplus 0\oplus 0\oplus 0\oplus M_1,\\
\label{equ:decomposition of 2I-P-Q}&&U_{P,Q}(2I-P-Q)U_{P,Q}^*=0\oplus I\oplus I\oplus 2I\oplus M_2,\nonumber\\
\label{equ:decomposition of 2I-P-Q+inverse of n}&&U_{P,Q}\big[2I-P-Q+\frac{I}{n}\big]^{-1}U_{P,Q}^*=nI\oplus \frac{nI}{n+1}\oplus\frac{nI}{n+1}\oplus \frac{nI}{2n+1}\oplus X_n^{-1},\nonumber
\end{eqnarray*}
where
\begin{align*}\label{equ:expression of M1}M_1&=\left(\begin{array}{cc}Q_0&Q_0^\frac{1}{2}(I-Q_0)^\frac{1}{2}U_0\\0&0\end{array}\right),\\
M_2&=\left(\begin{array}{cc}I-Q_0&-Q_0^\frac{1}{2}(I-Q_0)^\frac{1}{2}U_0\\-U_0^* Q_0^\frac{1}{2}(I-Q_0)^\frac{1}{2}&I+U_0^*Q_0U_0\end{array}\right),\nonumber\\
X_n&=\left(\begin{array}{cc}\frac{n+1}{n}I-Q_0&-Q_0^\frac{1}{2}(I-Q_0)^\frac{1}{2}U_0\\-U_0^* Q_0^\frac{1}{2}(I-Q_0)^\frac{1}{2}&\frac{n+1}{n}I+U_0^*Q_0U_0\end{array}\right).\nonumber
\end{align*}
Using the above equations and \eqref{eqn:defn of angle A n}, we obtain
\begin{equation}\label{equ:simplified expression of A n}A_n=0\oplus 0\oplus 0\oplus 0\oplus M_1M_2X_n^{-1}.\end{equation}
Similarly, we have
\begin{eqnarray*}\label{equ:introduce N1}&&U_{P,Q}(I-P)(I-Q)U_{P,Q}^*=0\oplus 0\oplus 0\oplus I\oplus N_1,\\
&&U_{P,Q}(P+Q)U_{P,Q}^*=2I\oplus I\oplus I\oplus 0\oplus N_2,\nonumber\\
&&U_{P,Q}\big[P+Q+\frac{I}{n}\big]^{-1}U_{P,Q}^*=\frac{nI}{2n+1}\oplus \frac{nI}{n+1}\oplus\frac{nI}{n+1}\oplus nI\oplus Y_n^{-1},\nonumber
\end{eqnarray*}
where
\begin{align*}\label{equ:expression of N1}N_1&=\left(\begin{array}{cc}0&0\\-U_0^* Q_0^\frac{1}{2}(I-Q_0)^\frac{1}{2}&U_0^*Q_0U_0\end{array}\right),\\
N_2&=\left(\begin{array}{cc}I+Q_0&Q_0^\frac{1}{2}(I-Q_0)^\frac{1}{2}U_0\\U_0^* Q_0^\frac{1}{2}(I-Q_0)^\frac{1}{2}&I-U_0^*Q_0U_0\end{array}\right),\nonumber\\
Y_n&=\left(\begin{array}{cc}\frac{n+1}{n}I+Q_0&Q_0^\frac{1}{2}(I-Q_0)^\frac{1}{2}U_0\\U_0^* Q_0^\frac{1}{2}(I-Q_0)^\frac{1}{2}&\frac{n+1}{n}I-U_0^*Q_0 U_0\end{array}\right).\nonumber
\end{align*}
Using the above equations and \eqref{eqn:defn of angle B n}, we obtain
\begin{equation}\label{equ:simplified expression of B n}B_n=0\oplus 0\oplus 0\oplus 0\oplus N_1N_2Y_n^{-1}.\end{equation}

Let $\widetilde{U_0}\in\mathcal{L}(H_5\oplus H_6)$ be the unitary defined by
\begin{eqnarray*}\widetilde{U_0}=\left(
 \begin{array}{cc}
 0 & U_0 \\
 -U_0^* & 0 \\
 \end{array}
 \right).
\end{eqnarray*}
Then $\widetilde{U_0}^*=-\widetilde{U_0}$ and
\begin{equation}\label{equ:seperated relationship}\widetilde{U_0}M_1\widetilde{U_0}^*=N_1,\quad \widetilde{U_0}M_2\widetilde{U_0}^*=N_2, \quad \widetilde{U_0}X_n\widetilde{U_0}^*=Y_n.
\end{equation}
Hence $\widetilde{U_0}X_n^{-1}\widetilde{U_0}^*=Y_n^{-1}$. Therefore, by \eqref{equ:simplified expression of A n}--\eqref{equ:seperated relationship}, we have
\begin{equation}\label{equ:unitary equivalent of A n and B n}\widetilde{U}A_n \widetilde{U}^*=B_n \ \mbox{for each $n\in\mathbb{N}$}, \end{equation}
where $\widetilde{U}$ is the unitary defined by
\begin{eqnarray*}\widetilde{U}=I_{\oplus_{i=1}^4 H_i}\oplus \widetilde{U_0}\in \mathcal{L}\Big(\oplus_{i=1}^6 H_i\Big).\end{eqnarray*}
In view of \eqref{eqn:defn of angle A n}, \eqref{eqn:defn of angle B n}, \eqref{equ:unitary equivalent of A n and B n} and Lemma~\ref{lem:the strict topology of the projection of the range-1}, we observe that for every $x\in H$,
\begin{align*}\Vert (I-P)(I-Q)P_{\overline{\mathcal{R}(P+Q)}}\,x\Vert&=\lim_{n\to\infty}\Vert U_{P,Q}^* B_n U_{P,Q}x\Vert\\
&=\lim_{n\to\infty}\Vert U_{P,Q}^* \widetilde{U} A_n \widetilde{U}^*U_{P,Q}x\Vert\\
&=\lim_{n\to\infty}\Vert A_n \widetilde{U}^*U_{P,Q}x\Vert\\
&=\Vert P QP_{\overline{\mathcal{R}(2I-P-Q)}}\,U^*_{P,Q}\widetilde{U}^*U_{P,Q}x\Vert,
\end{align*}
which gives \eqref{equ:norm equivalent for the angle-1}, since $U^*_{P,Q}\widetilde{U}^*U_{P,Q}$ is unitary and $x\in H$ is arbitrary.
\end{proof}

\begin{cor}\label{cor:application of Halmos two projections-1} Equation \eqref{equ:motivation equation} is true for every two projections $P$ and $Q$ on a Hilbert space.
\end{cor}
\begin{proof}Since any two projections on a Hilbert space are harmonious, the conclusion follows immediately from Lemma~\ref{lem:application of Halmos two projections-1}.
\end{proof}

Lemma~\ref{lem:application of Halmos two projections-1} can in fact be improved. To this end, we need a lemma as follows.

\begin{lem}\label{lem:technical lemma+1} Let $T\in\mathcal{L}(H)$ be positive such that $\overline{\mathcal{R}(T)}$ is orthogonally complemented in $H$. For every $n\in \mathbb{N}$, let $T_n$ be defined as in Lemma~\ref{lem:the strict topology of the projection of the range-1}. Then
\begin{equation}\label{equ:technical lemma+1}\lim_{n\to\infty}\|ST_n\|=\|SP_{\overline{\mathcal{R}(T)}}\|\ \mbox{for every $S\in \mathcal{L}(H)$}.\end{equation}
\end{lem}
\begin{proof} Clearly, $0\le T_n^2\le T_{n+1}^2\le P_{\overline{\mathcal{R}(T)}}$ for each $n\in\mathbb{N}$, so
$$0\le S T_n^2 S^*\le S T_{n+1}^2 S^*\le SP_{\overline{\mathcal{R}(T)}} S^*,$$
which gives by \cite[Proposition~1.3.5]{Pedersen} that
$$\| ST_n\|^2\le \|ST_{n+1}\|^2\le \|SP_{\overline{\mathcal{R}(T)}}\|^2\ \mbox{for each $n\in\mathbb{N}$}.$$
Hence $\lim\limits_{n\to\infty}\|ST_n\|$ exists and $\lim\limits_{n\to\infty}\|ST_n\|\le \|SP_{\overline{\mathcal{R}(T)}}\|$. On the other hand, given any $x\in H$, by Lemma~\ref{lem:the strict topology of the projection of the range-1} we have
\begin{align*}&\|SP_{\overline{\mathcal{R}(T)}}x\|=\lim_{n\to\infty} \|ST_n x\|\le \big(\lim\limits_{n\to\infty}\|ST_n\|\big)\|x\|,
\end{align*}
therefore $\|SP_{\overline{\mathcal{R}(T)}}\|\le \lim\limits_{n\to\infty}\|ST_n\|$. This completes the proof of \eqref{equ:technical lemma+1}.
\end{proof}

A generalization of Lemma~\ref{lem:application of Halmos two projections-1} is as follows.
\begin{lem}\label{lem:the main result+2}Let $P, Q\in\mathcal{L}(H)$ be two projections such that $\overline{\mathcal{R}(P+Q)}$ and $\overline{\mathcal{R}(2I-P-Q)}$
are both orthogonally complemented in $H$. Then equation \eqref{equ:motivation equation} is valid.
\end{lem}
\begin{proof} It needs only to prove that \eqref{equ:norm equivalent for the angle-1} is true. Since $\mathcal{L}(H)$ is a unital $C^*$-algebra, there exists a Hilbert space $E$ and a $C^*$-morphism $\pi:\mathcal{L}(H)\to \mathcal{L}(E)$ such that $\pi$ is faithful \cite[Corollary~3.7.5]{Pedersen}. Replacing $E$ with $\pi(I_{H})E$ if necessary, we may assume that $\pi$ is unital. For every $n\in\mathbb{N}$, let
\begin{align*}&X_n=(2I-P-Q)\big(2I-P-Q+\frac{1}{n}I\big)^{-1}, \\
&Y_n=\pi(X_n)=\big(2I-\pi(P)-\pi(Q)\big)\big(2I-\pi(P)-\pi(Q)+\frac{1}{n}I\big)^{-1},\\
&Z_n=(P+Q)(P+Q+\frac{1}{n}I)^{-1}, \\
&W_n=\pi(Z_n)=\big(\pi(P)+\pi(Q)\big)\big(\pi(P)+\pi(Q)+\frac{1}{n}I\big)^{-1}.
\end{align*}
Then according to Lemma~\ref{lem:technical lemma+1} and Corollary~\ref{cor:application of Halmos two projections-1}, we have
\begin{align*}\Vert P QP_{\overline{\mathcal{R}(2I-P-Q)}}\Vert&=\lim_{n\to\infty}\|P Q X_n\|=\lim_{n\to\infty}\|\pi(P)\pi(Q)Y_n\| \\
&=\Big\|\pi(P)\pi(Q)P_{\overline{\mathcal{R}\big(2I-\pi(P)-\pi(Q)\big)}}\Big\|\\
&=\Big\|\big(I-\pi(P)\big)\big(I-\pi(Q)\big)P_{\overline{\mathcal{R}\big(\pi(P)+\pi(Q)\big)}}\Big\|\\
&=\lim_{n\to\infty} \Big\|\big(I-\pi(P)\big)\big(I-\pi(Q)\big)W_n\Big\|\\
&=\lim_{n\to\infty} \|(I-P)(I-Q)Z_n\|\\
&=\|(I-P)(I-Q)P_{\overline{\mathcal{R}(P+Q)}}\|.
\end{align*}
The proof of \eqref{equ:norm equivalent for the angle-1} is then finished.
\end{proof}

It is remarkable that Lemma~\ref{lem:the main result+2} above can be generalized furthermore. Indeed, we will prove that equation \eqref{equ:motivation equation} is always true whenever
$\mathcal{R}(P)\cap\mathcal{R}(Q)$ and $\mathcal{N}(P)\cap\mathcal{N}(Q)$ are both orthogonally complemented in $H$. To this end, we need a couple of lemmas.

Recall that the Moore-Penrose inverse $T^\dag$ of an operator $T\in \mathcal{L}(H,K)$  is the
unique element $X\in \mathcal{L}(K,H)$ which satisfies
\begin{equation*} \label{equ:m-p inverse} TXT=T,\quad XTX=X,\quad (TX)^*=TX, \quad (XT)^*=XT.\end{equation*}
If such an operator  $T^\dag$ exists, then  $T$ is said to be M-P invertible.

\begin{lem}\label{lem:Xu-Sheng condition of M-P invertible}{\rm \cite[Theorem~2.2]{Xu-Sheng}}\ For every $T\in\mathcal{L}(H,K)$, $T$ is M-P invertible if and only if $\mathcal{R}(T)$ is closed.
\end{lem}

\begin{rem}\label{rem:closeness implies orthogonality}{\rm  Let $T\in\mathcal{L}(H,K)$ be such that $\mathcal{R}(T)$ is closed. Then
both $P=T T^\dag $ and $Q=T^\dag T$  are projections such that  $\mathcal{R}(P)=\mathcal{R}(T)$ and $\mathcal{R}(Q)=\mathcal{R}(T^*)$. So in this case,
$\mathcal{R}(T)$ and $\mathcal{R}(T^*)$ are orthogonally complemented in $K$ and $H$, respectively.
}\end{rem}

\begin{lem}\label{lem:orthogonal} {\rm (cf.\,\cite[Theorem 3.2]{Lance} and \cite[Remark 1.1]{Xu-Sheng})}\ Let $T\in \mathcal{L}(H,K)$. Then the closedness of any one of the following sets
implies the closedness of the remaining three sets:
$$\mathcal{R}(T),\quad \mathcal{R}(T^*),\quad  \mathcal{R}(TT^*), \quad \mathcal{R}(T^*T).$$
If $\mathcal{R}(T)$ is closed, then $\mathcal{R}(T)=\mathcal{R}(TT^*)$,
$\mathcal{R}(T^*)=\mathcal{R}(T^*T)$ and the following  orthogonal
decompositions hold:
\begin{equation*}\label{equ:orthogonal decomposition} H=\mathcal{R}(T^*)\dotplus \mathcal{N}(T), \quad K=\mathcal{R}(T)\dotplus \mathcal{N}(T^*).\end{equation*}
\end{lem}

\begin{rem}\label{rem:ranges of P plus Q equals-1}{\rm
Suppose that $P,Q\in\mathcal{L}(H)$ are projections.  Let $T$ be defined by \eqref{equ:inclusion of two projectons}. Then it follows from \eqref{equ:relationship of the ranges of two operators-1} and Lemma~\ref{lem:orthogonal} that
\begin{equation}\label{equ:closeness of the ranges of two projections-plus}\mathcal{R}(P)+\mathcal{R}(Q)\ \mbox{is closed if and only if}\ \mathcal{R}(P+Q)\ \mbox{is closed},\end{equation}
and in this case  $\mathcal{R}(P)+\mathcal{R}(Q)=\mathcal{R}(P+Q)$.
}\end{rem}

\begin{lem}\label{lem:summation of two projections has closed range}{\rm \cite[Proposition~4.6]{Luo-Song-Xu}}\ Let $P,Q\in\mathcal{L}(H)$ be projections such that $\mathcal{R}(P)+\mathcal{R}(Q)$ is closed. Then
\begin{equation}\label{equ:conjecture is true under the closeness assumption}\mathcal{R}(I-P)+\mathcal{R}(I-Q)=\big(\mathcal{R}(P)\cap \mathcal{R}(Q)\big)^\perp.\end{equation}
\end{lem}

We provide a technical lemma of this section as follows.
\begin{lem}\label{lem:norm of two projections less than one} Let $P,Q\in\mathcal{L}(H)$ be projections. Then the following statements are equivalent:
\begin{enumerate}
\item[{\rm (i)}] $\Vert PQ\Vert<1$;
\item[{\rm (ii)}] $\mathcal{R}(P)\cap \mathcal{R}(Q)=\{0\}$ and $\mathcal{R}(P)+\mathcal{R}(Q)$ is closed;
\item[{\rm (iii)}] $\mathcal{R}(I-P)+\mathcal{R}(I-Q)=H$.
\end{enumerate}
\end{lem}
\begin{proof}``(i)$\Longrightarrow$(ii)":  Assume that $\Vert PQ\Vert<1$. If $\mathcal{R}(P)\cap \mathcal{R}(Q)\neq\{0\}$, then there exists $x_0\in \mathcal{R}(P)\cap \mathcal{R}(Q)$ such that $\Vert x_0\Vert=1$.  Thus,
$1>\Vert PQ\Vert\geq \Vert PQx_0\Vert=1$, which is a contradiction. Therefore, $\mathcal{R}(P)\cap \mathcal{R}(Q)=\{0\}$.

Given any $y\in  \overline{\mathcal{R}(P)+\mathcal{R}(Q)}$, there exist sequences $\{x_n\}$ and $\{y_n\}$ in $H$ such that
\begin{equation}\label{equ:element in the closure-1}Px_n+Qy_n\longrightarrow y,\end{equation}
which gives
\begin{equation*}(I-PQ)Qy_n=(I-P)(Px_n+Qy_n)\longrightarrow (I-P)y.
\end{equation*}
Since $\Vert PQ\Vert< 1$, the operator $I-PQ$ is invertible. Hence
\begin{equation*}Qy_n\longrightarrow (I-PQ)^{-1}(I-P)y=Qu\ \mbox{for some $u\in H$}. \end{equation*}
Taking limits together with
\eqref{equ:element in the closure-1} yield
\begin{equation*}Px_n=(Px_n+Qy_n)-Qy_n\longrightarrow y-Qu=Pv\ \mbox{for some $v\in H$},\end{equation*}
from which we get $y=Pv+Qu\in \mathcal{R}(P)+\mathcal{R}(Q)$. This completes the proof of the closedness of $\mathcal{R}(P)+\mathcal{R}(Q)$.

``(ii)$\Longrightarrow$(i)": From $\Vert PQ\Vert^2=\Vert PQQ^*P^*\Vert=\Vert PQP\Vert$ and the positivity of $PQP$, we conclude that
\begin{equation*}\Vert PQ\Vert<1\Longleftrightarrow \Vert PQP\Vert<1 \Longleftrightarrow I-PQP  \ \mbox{is invertible}.
\end{equation*}
Assume that $\mathcal{R}(P)\cap \mathcal{R}(Q)=\{0\}$ and $\mathcal{R}(P)+\mathcal{R}(Q)$ is closed. We show that
$ I-PQP$ is invertible.

Injectivity: Let $x\in H$ be such that $(I-PQP)x=0$. Then $x=PQPx\in\mathcal{R}(P)$, so $x=Px$ and thus
$P(I-Q)Px=0$. Hence $(I-Q)Px=0$. Therefore,
$x=Px=QPx\in \mathcal{R}(P)\cap \mathcal{R}(Q)=\{0\}$.

Surjectivity:  By \eqref{equ:closeness of the ranges of two projections-plus}, we know that  $\mathcal{R}(P+Q)$ is closed and $\mathcal{R}(P)+\mathcal{R}(Q)=\mathcal{R}(P+Q)$. Given any $y\in\overline{\mathcal{R}\big((I-Q)P\big)}$,
there exists a sequence $\{x_n\}$ in $H$ such that
\begin{equation*}Px_n+Q(-Px_n)=(I-Q)Px_n\longrightarrow y=(I-Q)y=(P+Q)w\ \mbox{for some $w\in H$},\end{equation*}
since both $\mathcal{R}(I-Q)$ and $\mathcal{R}(P+Q)$ are closed. Then
\begin{equation*}y=(I-Q)(I-Q)y=(I-Q)(P+Q)w=(I-Q)Pw\in \mathcal{R}\big((I-Q)P\big).
\end{equation*}
The process above shows that $\mathcal{R}\big((I-Q)P\big)$ is closed. In view of Lemma~\ref{lem:orthogonal}, we infer that
$\mathcal{R}\big(P(I-Q)P\big)=\mathcal{R}\Big(\big((I-Q)P\big)^*(I-Q)P\Big)$ is also closed.

Clearly, $\mathcal{N}(P)\subseteq \mathcal{N}\big(P(I-Q)P\big)$. Conversely, given any $x\in H$ such that $P(I-Q)Px=0$,
we arrive at $(I-Q)Px=0$. Hence
$$Px=QPx\in \mathcal{R}(P)\cap \mathcal{R}(Q)=\{0\}.$$ Therefore,
$\mathcal{N}\big(P(I-Q)P\big)\subseteq \mathcal{N}(P)$. This completes the proof that $\mathcal{N}(P)=\mathcal{N}\big(P(I-Q)P\big)$.

Accordingly, by Lemma~\ref{lem:orthogonal}, $H$ can be orthogonally decomposed  as
\begin{equation*}H=\mathcal{R}\big(P(I-Q)P\big)\dotplus \mathcal{N}(P).\end{equation*}

Now, given any $x\in H$, there exist $u\in H$ and $v\in\mathcal{N}(P)$ such that
$$x=P(I-Q)Pu+v,$$
so that $Px=P(I-Q)Pu$. Therefore,
\begin{align*}(I-PQP)\big(Pu+(I-P)x\big)=P(I-Q)Pu+(I-P)x=x.
\end{align*}
This completes the proof that $\mathcal{R}(I-PQP)=H$.

``(ii)$\Longrightarrow$(iii)": The conclusion is directly deduced  from \eqref{equ:conjecture is true under the closeness assumption} in Lemma~\ref{lem:summation of two projections has closed range}.

``(iii)$\Longrightarrow$(ii)": Assume that $\mathcal{R}(I-P)+\mathcal{R}(I-Q)=H$. Replacing $P$ and $Q$ in Lemma~\ref{lem:summation of two projections has closed range} with $I-P$ and $I-Q$, respectively, we infer that $\mathcal{R}(P)+\mathcal{R}(Q)$ is closed. From \eqref{equ:conjecture is true under the closeness assumption} we get
$\big(\mathcal{R}(P)\cap\mathcal{R}(Q)\big)^\bot=H$, which can happen only if $\mathcal{R}(P)\cap\mathcal{R}(Q)=\{0\}$.
\end{proof}

\begin{lem}\label{lem:4-equivalent conditions-1}Let $P,Q\in\mathcal{L}(H)$ be projections such that
$\mathcal{R}(P)\cap \mathcal{R}(Q)$ is orthogonally complemented in $H$.
Then the following statements are equivalent:
\begin{enumerate}
\item[{\rm (i)}] $\Vert PQ-P_{\mathcal{R}(P)\cap \mathcal{R}(Q)}\Vert<1$;
\item[{\rm (ii)}] $\mathcal{R}(P)\cap \big(\mathcal{R}(P)\cap \mathcal{R}(Q)\big)^\bot+\mathcal{R}(Q)\cap \big(\mathcal{R}(P)\cap \mathcal{R}(Q)\big)^\bot$ is closed;
\item[{\rm (iii)}] $\mathcal{R}(P)+\mathcal{R}(Q)$ is closed;
\item[{\rm (iv)}] $\mathcal{R}(I-P)+\mathcal{R}(I-Q)$ is closed.
\end{enumerate}
\end{lem}
\begin{proof}For simplicity, we put $\mathcal{R}(P)\cap \mathcal{R}(Q)=\Omega$. It is obvious that both $P$ and $Q$ commute with $I-P_\Omega$. Hence, if we put
\begin{equation}\label{equ:defn of P1 and Q1}P_1=P(I-P_\Omega), \quad Q_1=Q(I-P_\Omega),\end{equation} then
$P_1$ and $Q_1$ are projections such that
\begin{equation}\label{equ:ranges of P1 and Q1}\mathcal{R}(P_1)=\mathcal{R}(P)\cap \Omega^\bot, \quad \mathcal{R}(Q_1)=\mathcal{R}(Q)\cap \Omega^\bot.
\end{equation}
Furthermore, it is clear that
\begin{equation}\label{equ:intersection of the ranges is zero}PQ-P_\Omega=PQ(I-P_\Omega)=P_1Q_1,\quad \mathcal{R}(P_1)\cap\mathcal{R}(Q_1)=\Omega\cap \Omega^\bot=\{0\}.\end{equation}

``(i)$\Longleftrightarrow$(ii)":\ From \eqref{equ:ranges of P1 and Q1}, \eqref{equ:intersection of the ranges is zero} and Lemma~\ref{lem:norm of two projections less than one},  we conclude that \begin{eqnarray*}\Vert PQ-P_\Omega\Vert<1&\Longleftrightarrow& \Vert P_1Q_1\Vert<1\\
&\Longleftrightarrow& \mathcal{R}(P_1)+\mathcal{R}(Q_1)\ \mbox{is closed}\\
&\Longleftrightarrow& \mbox{$\mathcal{R}(P)\cap \Omega^\bot+\mathcal{R}(Q)\cap \Omega^\bot$ is closed.}
\end{eqnarray*}

``(ii)$\Longleftrightarrow$(iii)": Let $P_1, Q_1$ be defined by \eqref{equ:defn of P1 and Q1}, and put
\begin{equation}\label{equ:defn of T temporary}T=P_1+Q_1=(P+Q)(I-P_\Omega)=(I-P_\Omega)(P+Q)=P+Q-2P_\Omega.\end{equation} Since both $P_1$ and $Q_1$ are projections, by
\eqref{equ:closeness of the ranges of two projections-plus} and \eqref{equ:ranges of P1 and Q1}, we see that
$\mathcal{R}(P)\cap \Omega^\bot+\mathcal{R}(Q)\cap \Omega^\bot$ is closed if and only $\mathcal{R}(T)$ is closed.

Suppose that $\mathcal{R}(T)$ is closed. Given any $x\in \overline{\mathcal{R}(P)+\mathcal{R}(Q)}=\overline{\mathcal{R}(P+Q)}$, there exist $u_n\in H\,\,(n\in\mathbb{N})$ such that
$(P+Q)u_n\to x$ as $n\to\infty$. Then
$$Tu_n=(I-P_\Omega)(P+Q)u_n\longrightarrow (I-P_\Omega)x=Tu\ \mbox{for some $u\in H$}.$$
Now \eqref{equ:defn of T temporary} ensures that
\begin{align*}x&=P_\Omega x+(I-P_\Omega) x=P_\Omega x+Tu=P_\Omega x+(P+Q)u-2P_\Omega u\\
&=P\big(u+P_\Omega(x-2u)\big)+Qu\in \mathcal{R}(P)+\mathcal{R}(Q).
\end{align*}
This completes the proof of the closedness of $\mathcal{R}(P)+\mathcal{R}(Q)$.

Conversely, suppose that $\mathcal{R}(P)+\mathcal{R}(Q)$ is closed. By \eqref{equ:closeness of the ranges of two projections-plus},
$\mathcal{R}(P+Q)$ is also closed. Given any $x\in\overline{\mathcal{R}(T)}$, there exist $u_n\in H (n\in\mathbb{N})$ such that
$Tu_n\to x$ as $n\to\infty$. Then
\begin{equation*}x=\lim_{n\to\infty}Tu_n=\lim_{n\to\infty}(I-P_\Omega)Tu_n=(I-P_\Omega)x.\end{equation*}
Meanwhile, since $\mathcal{R}(P+Q)$ is closed, from \eqref{equ:defn of T temporary} we get
$$\mbox{$x=(P+Q)u$ for some $u\in H$.}$$
It follows that $x=(I-P_\Omega)(P+Q)u=Tu$. This completes the proof of the closedness of $\mathcal{R}(T)$.

``(iii)$\Longleftrightarrow$(iv)": The conclusion is directly deduced from Lemma~\ref{lem:summation of two projections has closed range}.
\end{proof}

Now, we are in the position to give  to the main result of this section as follows.
\begin{thm}\label{thm:the final main result}Let $P, Q\in\mathcal{L}(H)$ be two projections such that $\mathcal{R}(P)\cap\mathcal{R}(Q)$ and $\mathcal{N}(P)\cap\mathcal{N}(Q)$ are both orthogonally complemented in $H$. Then equation \eqref{equ:motivation equation} is valid.
\end{thm}
\begin{proof}Two cases are to be taken into consideration.

\textbf{Case 1}\quad $\mathcal{R}(P)+\mathcal{R}(Q)$ is closed. In this case,  $\mathcal{R}(I-P)+\mathcal{R}(I-Q)$ is also closed by Lemma~\ref{lem:4-equivalent conditions-1}. Therefore, by \eqref{equ:closeness of the ranges of two projections-plus} and Remark~\ref{rem:closeness implies orthogonality} we know that $\mathcal{R}(P+Q)$ and $\mathcal{R}(I-P+I-Q)$ are both orthogonally complemented in $H$, hence equation \eqref{equ:motivation equation} is valid by Lemma~\ref{lem:the main result+2}.

\textbf{Case 2}\quad $\mathcal{R}(P)+\mathcal{R}(Q)$ is not closed. In this case,  $\mathcal{R}(I-P)+\mathcal{R}(I-Q)$ is also not closed by Lemma~\ref{lem:4-equivalent conditions-1}. Therefore, by Lemma~\ref{lem:4-equivalent conditions-1} we conclude that
\begin{equation*}\big\Vert PQ\big(I-P_{\mathcal{R}(P)\cap \mathcal{R}(Q)}\big) \big\Vert=\big\Vert (I-P)(I-Q)\big(I-P_{\mathcal{N}(P)\cap \mathcal{N}(Q)}\big)\big\Vert=1.
\end{equation*}
So, in this case, equation \eqref{equ:motivation equation} is valid.
\end{proof}

\begin{rem}{\rm  Let $M$ and $N$ be two closed submodules of $H$ such that
$M\cap N$ is orthogonally complemented in $H$. As in the Hilbert space case, we can define the Friedrichs angle  $\alpha(M,N)$ through
$c(M,N)$, which is formulated by \eqref{equ:computation of C M N}. If furthermore $M^\bot\cap N^\bot$ is also orthogonally complemented in $H$, then
from \eqref{equ:computation of C M N} and Theorem~\ref{thm:the final main result}  we can conclude that the characterization \eqref{equ:M,N norm equivalent for the angle-1} of the Friedrichs angle is also true.
}\end{rem}

\textbf{Acknowledgement.} The authors would like to sincerely thank the anonymous referee for carefully reading the paper and for useful comments.


\bigskip

\begin{thebibliography}{99}

\bibitem{ACL}E.  Andruchow, E. Chiumiento and G. Larotonda, Geometric significance of Toeplitz kernels, J. Funct. Anal. 275 (2018), no. 2, 329--355.

\bibitem{AC} E. Andruchow and G. Corach, Essentially orthogonal subspaces, J. Operator Theory 79 (2018), no. 1, 79--100.


\bibitem{Bottcher-Spitkovsky}A. B\"ottcher and I. M. Spitkovsky, A gentle guide to the basics of two projections theory, Linear Algebra Appl. 432 (2010), no. 6, 1412--1459.

\bibitem{BRO}L. G. Brown, Nearly relatively compact projections in operator algebras, Banach J. Math. Anal. 12 (2018), no. 2, 259--293.

\bibitem{Deng-Du}C. Deng and H. Du, Common complements of two subspaces and an answer to Gro{\ss}'s question, Acta Mathematica Sinica (Chinese Series) 49 (2006), no. 5, 1099--1112.


\bibitem{Deutsch}F. Deutsch, The angle between subspaces of a Hilbert space, Approximation theory, wavelets and applications (Maratea, 1994), 107--130, NATO Adv. Sci. Inst. Ser. C Math. Phys. Sci., 454, Kluwer Acad. Publ., Dordrecht, 1995.

\bibitem{Foulis-Jencov¨¢-Pulmannov¨¢}D. J. Foulis, A. Jencov¨¢ and ¨¢ S. Pulmannov, A projection and an effect in a synaptic algebra, Linear Algebra Appl. 485 (2015), 417--441.


\bibitem{FRA} M. Frank, Geometrical aspects of Hilbert $C^*$-modules, Positivity 3 (1999), no. 3, 215--243.

\bibitem{Friedrichs} K. Friedrichs, On certain inequalities and characteristic value problems for analytic functions and for functions of two variables, Trans. Amer. Math. Soc. 41 (1937), no. 3,  321--364.

\bibitem{Halmos} P. Halmos, Two subspaces, Trans. Amer. Math. Soc. 144 (1969), 381--389.

\bibitem{Lance}E. C. Lance, Hilbert $C^*$-modules--A toolkit for operator algebraists, Cambridge University Press, Cambridge, 1995.

\bibitem{Liu-Luo-Xu}N. Liu, W. Luo and Q. Xu, The polar decomposition for adjointable operators on Hilbert $C^*$-modules and centered operators, Adv. Oper. Theory 3 (2018), no. 4, 855--867.

\bibitem{Luo-Song-Xu} W. Luo, C. Song and Q. Xu, The parallel sum  for adjointable operators on Hilbert $C^*$-modules, Acta Math. Sinica (Chinese Series), to appear. http://arxiv.org/abs/1806.04227v3.

\bibitem{MT1}  V. M. Manuilov and K. Thomsen, Shape theory and extensions of $C^*$-algebras, J. London. Math. Soc. (2) 84 (2011), no. 1, 183--203.

\bibitem{Manuilov-Moslehian-Xu} V. M. Manuilov, M. S. Moslehian and Q. Xu, Solvability of the equation $AX=C$ for operators on Hilbert $C^*$-modules, preprint, http://arxiv.org/abs/1807.00579v2.

\bibitem{MT2} V. M. Manuilov and E. V. Troitsky, Hilbert $C^*$-modules, Translated from the 2001 Russian original by the authors. Translations of Mathematical Monographs, 226. American Mathematical Society, Providence, RI, 2005.

\bibitem{MKX} M. S. Moslehian, M. Kian and Q. Xu, Positivity of $2\times 2$ block matrices of operators, Banach J. Math. Anal., in press, arxiv: https://arxiv.org/abs/1904.08680

\bibitem{Pedersen} G. K. Pedersen, $C^*$-algebras and their
automorphism groups (London Math. Soc. Monographs 14), Academic
Press, New York, 1979.

\bibitem{SIM}B. Simon, Unitaries permuting two orthogonal projections, Linear Algebra Appl. 528 (2017), 436--441.

\bibitem{Uebersohn}C. Uebersohn, On the difference of spectral projections, Integral Equations Operator Theory  90 (2018), no. 4, Art. 48, 30 pp.


\bibitem{Wegge-Olsen}N. E. Wegge-Olsen, $K$-theory and $C^*$-algebras: A friendly approach,  Oxford Science Publications. The Clarendon Press, Oxford University Press, New York, 1993.


\bibitem{Xu-Sheng} Q. Xu and L. Sheng, Positive semi-definite matrices of adjointable operators on Hilbert $C^*$-modules, Linear Algebra Appl. 428 (2008), no. 4, 992--1000.

\end{thebibliography}
\end{document}